\documentclass[11pt]{article}
\usepackage{amsmath,enumerate,amsfonts,color,amssymb,amsthm, verbatim}

\usepackage{hyperref}
\usepackage[normalem]{ulem}

\def\RR{{\mathbb R}}
\def\Sphere{{\mathbb S}}
\def\mcA{{\mycal A}}
\def\mcB{{\mycal B}}
\def\mcH{{\mycal H}}

\def\tf{\tilde f}

\newcommand{\f}{\varphi}

\def\bS{{\mathbb S}}
\def\eps{{\varepsilon}}

\newtheorem{theorem} {\sc  Theorem\rm}

\newtheorem{conjecture}[theorem]{\sc  Open problem \rm}

\newtheorem{remark}[theorem]{\sc  Remark\rm}

\newcounter{marnote}

\DeclareFontFamily{OT1}{rsfs}{}
\DeclareFontShape{OT1}{rsfs}{m}{n}{ <-7> rsfs5 <7-10> rsfs7 <10-> rsfs10}{}
\DeclareMathAlphabet{\mycal}{OT1}{rsfs}{m}{n}

\newcommand{\R}{\mathbb{R}}

\def\be{\begin{equation}}
\def\ee{\end{equation}}
\def\bea#1\eea{\begin{align}#1\end{align}}

\begin{document}
\title{Vortex sheet solutions for the Ginzburg-Landau system\\ in cylinders: symmetry and global minimality}

\author{
{\Large Radu Ignat}\footnote{Institut de Math\'ematiques de Toulouse \& Institut Universitaire de France, UMR 5219, Universit\'e de Toulouse, CNRS, UPS
IMT, F-31062 Toulouse Cedex 9, France. Email: Radu.Ignat@math.univ-toulouse.fr} \and {\Large Mircea Rus}\footnote{Department of Mathematics, Technical University of Cluj-Napoca, 400027 Cluj-Napoca, Romania. Email: rus.mircea@math.utcluj.ro}
}

\maketitle

\begin{abstract}
We consider the Ginzburg-Landau energy $E_\eps$ for $\R^M$-valued maps defined in a cylinder shape domain $B^N\times (0,1)^n$ satisfying a degree-one vortex boundary condition on $\partial B^N\times (0,1)^n$ in dimensions $M\geq N\geq 2$ and $n\geq 1$. The aim is to study the radial symmetry of global minimizers of this variational problem. We prove the following: if $N\geq 7$, then for every $\eps>0$, there exists a unique global minimizer which is given by the non-escaping radially symmetric vortex sheet solution $u_\eps(x,z)=(f_\eps(|x|) \frac{x}{|x|}, 0_{\R^{M-N}})$, $\forall x\in B^N$ that is invariant in $z\in (0,1)^n$. If $2\leq N \leq 6$ and $M\geq N+1$, the following dichotomy occurs between escaping and non-escaping solutions: there exists $\eps_N>0$ such that

$\bullet$ if $\eps\in (0, \eps_N)$, then every global minimizer is an escaping radially symmetric vortex sheet solution of the form $R \tilde u_\eps$ where $\tilde u_\eps(x,z)=(\tf_{\eps}(|x|) \frac{x}{|x|}, 0_{\R^{M-N-1}}, g_{\eps}(|x|))$ is invariant in $z$-direction with $g_\eps>0$ in $(0,1)$ and $R\in O(M)$ is an orthogonal transformation keeping invariant the space $\R^N\times \{0_{\R^{M-N}}\}$;

$\bullet$ if $\eps\geq \eps_N$, then the non-escaping radially symmetric vortex sheet solution $u_\eps(x,z)=(f_\eps(|x|) \frac{x}{|x|}, 0_{\R^{M-N}})$, $\forall x\in B^N, z\in (0,1)^n$ is the unique global minimizer; moreover, there are no bounded escaping solutions in this case.

We also discuss the problem of  vortex sheet $\bS^{M-1}$-valued harmonic maps. 

\bigskip


\end{abstract}

\tableofcontents

\section{Introduction and main results}

In this paper, we consider the following Ginzburg-Landau type energy functional
\be
\label{en}
E_\eps(u)=\int_{\Omega} \Big[\frac{1}{2}|\nabla u|^2+\frac{1}{2\eps^2}W(1 - |u|^2)\Big]\,dX,
\ee
where $\eps>0$, $X=(x,z)\in \Omega=B^N\times (0,1)^n$ is a cylinder shape domain with $B^N$ the unit ball in $\R^N$, {$n\geq 1$, $N \geq 2$} and the potential $W\in C^2((-\infty,1]; \R)$ satisfies 
\begin{equation}
 W(0)=0,\, W(t)>0 \hbox{ for all } t\in (-\infty, 1]\setminus \{0\} {\text{ and }} W \textrm{ is convex}.
 	\label{Eq:26VI18-E1}
\end{equation}
(The prototype potential is $W(t)=\frac{t^2}2$ for $t\leq 1$.)  We investigate the global minimizers of the energy $E_\eps$ in the set of $\R^N$-valued maps:
$$
\mcA_N:=\{ u\in H^1(\Omega; \R^N):\,  u(x,z)=x \textrm{ for every } x\in \partial B^N = \bS^{N-1}, z\in (0,1)^n\}.
$$ 
The boundary assumption $u(x,z) = x$ for every  $x\in \bS^{N-1}$ and every $z\in (0,1)^n$ is referred in the literature as the degree-one vortex boundary condition.

The direct method in the calculus of variations yields the existence of a global minimizer $u_\eps$ of $E_\eps$ over $\mcA_N$ for all range of $\eps>0$. Moreover, any minimizer $u_\eps$ satisfies $|u_\eps| \leq 1$ in $\Omega$, $u_\eps$  belongs to $C^1(\overline{\Omega}; \R^N)$ and solves the system of PDEs (in the sense of distributions) with mixed Dirichlet-Neumann boundary conditions:
\be
\label{E-L}
\left\{\begin{array}{l}
-\Delta u_\eps=\frac1{\eps^2} u_\eps \, W'(1-|u_\eps|^2)\quad \textrm{ in } \, \Omega,\\
\frac{\partial u_\eps}{\partial z}=0 \quad \textrm{ on } \, B^N\times \partial (0,1)^n,\\
u(x,z)=x \quad \textrm{ on } \, \partial B^N\times (0,1)^n.
\end{array}
\right.
\ee

\subsection{Minimality of the $\R^N$-valued vortex sheet solution}
The first goal of this paper is to prove the uniqueness and radial symmetry of the global minimizer of $E_\eps$ in $\mcA_N$ for \emph{all} $\eps>0$ in dimensions $N \geq 7$ and $n\geq 1$. In fact, in these dimensions, we show that the global minimizer of $E_\eps$ in $\mcA_N$ is unique and given by the following radially symmetric critical point of $E_\eps$ that is invariant in $z$: \footnote{If $n=0$ and $N\geq 2$, then $SO(N)$ induces a group action on $\mcA_N$ given by $u(x)\mapsto R^{-1}u(Rx)$ for every $x\in B^N$, $R\in SO(N)$ and $u\in \mcA_N$ under which the energy $E_\eps$ and the vortex boundary condition are invariant. Then every bounded critical point of $E_\eps$ in $\mcA_N$ that is invariant under this $SO(N)$ group action has the form \eqref{def_sol_equa}, see e.g. \cite[Lemma A.4]{IN}.}
\be
\label{def_sol_equa}
u_\eps(x,z)=f_\eps(|x|)\frac{x}{|x|} \quad \textrm{ for all } x\in B^N \textrm{ and } z\in (0,1)^n,
\ee
where the radial profile $f_\eps : [0,1]\to \R$ in $r=|x|$ is the unique solution to the ODE:
\be
\label{1}
\left\{\begin{array}{l}
-f''_\eps-\frac{N-1}{r}f'_\eps+\frac{N-1}{r^2}f_\eps=\frac{1}{\eps^2} f_\eps \,W'(1-f_\eps^2) \quad 
\text{ for } r \in (0,1),\\
f_\eps(0)=0, f_\eps(1)=1.
\end{array}
\right.
\ee
We recall that the unique radial profile $f_\eps$ satisfies $f_\eps>0$ and $f'_\eps>0$ in $(0,1)$ (see e.g. \cite{HH, ODE_INSZ, IN}). Note that the zero set of $u_\eps$ is given by the $n$-dimensional vortex sheet $\{0_{\R^N}\}\times (0,1)^n$ in $\Omega$ (in particular, if $n=0$, it is a vortex point, while for  $n=1$, it is a vortex filament); therefore, $u_\eps$ in \eqref{def_sol_equa} is called (radially symmetric) {\it vortex sheet solution} to the Ginzburg-Landau system \eqref{E-L}.

\medskip

\begin{theorem}
\label{GL_equator}
Assume that $W$ satisfies \eqref{Eq:26VI18-E1} and $n\geq 1$. If $N\geq 7$, then $u_\eps$ given in \eqref{def_sol_equa} is the unique global minimizer of $E_\eps$  in $\mcA_N$ for every $\eps>0$.
\end{theorem}

The proof is reminiscent of the works of Ignat-Nguyen-Slastikov-Zarnescu \cite{INSZ_ENS, INSZ_CRAS}  studying uniqueness and symmetry of minimizers of the Ginzburg-Landau functionals for $\RR^M$-valued maps defined on smooth $N$-dimensional domains, where $M$ is not necessarily equal to  $N$.  
The idea is to analyze $E_\eps(u)$ for an arbitrary map $u$ and to exploit the convexity of $W$ to lower estimate the excess energy w.r.t. $E_\eps(u_\eps)$ by a suitable quadratic energy functional depending on $u-u_\eps$. This quadratic functional comes from the linearized PDE at $u_\eps$ and can be handled by a factorization argument. The positivity of the excess energy then follows by a Hardy-type inequality holding true  only in high dimensions $N \geq 7$. This is similar to the result of J\"ager and Kaul \cite{JagKaul} on the minimality of the equator map for the harmonic map problem in dimension $N \geq 7$ that is proved using a certain inequality involving the sharp constant in the Hardy inequality.

We expect that our result remains valid in dimensions $2 \leq N \leq 6$:

\begin{conjecture}
\label{conj}
Assume that $W$ satisfies \eqref{Eq:26VI18-E1}, $n\geq 1$ and $2\leq N\leq 6$. Is it true that for every $\eps>0$, $u_\eps$ given in \eqref{def_sol_equa} is the unique global minimizer of $E_\eps$  in $\mcA_N$?
\end{conjecture}

It is well known that the uniqueness of $u_\eps$ holds true for large enough $\eps>0$ in any dimension $N \geq 2$. Indeed, denoting by $\lambda_1$  the first eigenvalue of $-\Delta_x$ in $B^N$ with  zero Dirichlet boundary condition, then for any $\eps > \sqrt{W'(1)/\lambda_1}$, $E_\eps$ is strictly convex in $\mcA_N$ (see e.g., \cite[Theorem VIII.7]{BBH}, \cite[Remark 3.3]{INSZ_ENS}) and thus has a unique critical point in $\mcA_N$ that is the global minimizer of our problem. We improve this result as follows: for the radial profile $f_\eps$ in \eqref{1}, we denote by $\ell(\eps)$ the first eigenvalue of the operator
\begin{equation}
L_\eps = -\Delta_x - \frac{1}{\eps^2}W'(1 - f_\eps^2)
	\label{GL-op}
\end{equation}
acting on maps defined in $B^N$ with zero Dirichlet boundary condition. 
It is proved in \cite[Lemma 2.3]{IN} that if $2 \leq N \leq 6$ and $W \in C^2((-\infty,1])$ satisfies \eqref{Eq:26VI18-E1}, then the first eigenvalue $\ell(\eps)$ is a continuous function in $\eps$  
and there exists $\eps_N \in (0,\infty)$ such that 
\be
\label{epsN}
\textrm{$\ell(\eps) < 0$ in $(0,\eps_N)$, $\quad \ell(\eps_N) = 0\quad $ and $\quad \ell(\eps) > 0$ in $(\eps_N,\infty)$}.
\ee
Note that\footnote{Indeed, if $v\in H^1_0(B^N)$ is a first eigenfunction of $L_{\eps_N}$ in $B^N$ such that $\|v\|_{L^2(B^N)}=1$ then 
$$\lambda_1\leq \int_{B^N} |\nabla_x v|^2\, dx=\frac1{\eps_N^2}\int_{B^N} W'(1-f_{\eps_N}^2) v^2\, dx<\frac{W'(1)}{\eps_N^2}$$
because $\ell(\eps_N)=0$, $0<f_{\eps_N}<1$ in $(0,1)$ and \eqref{Eq:26VI18-E1} implies $W'(0)=0$ and $W'(t)>0$ for $t\in (0,1]$.} $0=\ell(\eps_N)> \lambda_1-\frac1{\eps_N^2} W'(1)$ yielding $$\eps_N<\sqrt{W'(1)/\lambda_1}.$$ 

\begin{theorem}
\label{eps_large}
Assume that $W$ satisfies \eqref{Eq:26VI18-E1}, $n\geq 1$ and $2\leq N\leq 6$. If $\eps\geq \eps_N$, then $u_\eps$ given in \eqref{def_sol_equa} is a global minimizer of $E_\eps$  in $\mcA_N$. Moreover, if either $\eps>\eps_N$, or $(\eps=\eps_N$ and $W$ is in addition strictly convex), then $u_\eps$ is the unique global minimizer of $E_\eps$  in $\mcA_N$. 
\end{theorem}

The case $\eps<\eps_N$ is still not solved as stated in  Open Problem \ref{conj}. 
Let us summarize some known results:

\medskip

\noindent {\it I. The case of $n=0$ and $\Omega=B^N$ (we also discuss here the problem for 
$\Omega= \R^N$)}. In this case, the above question was raised in dimension $N = 2$ for the disk $\Omega=B^2$ in the seminal book of Bethuel, Brezis and H\'elein \cite[Problem 10, page 139]{BBH}, and in general dimensions $N \geq 2$ and also for the  blow-up limiting problem around the vortex point (when the domain $\Omega$ is the whole space $\R^N$ and by rescaling, $\eps$ can be assumed equal to $1$) in an article of Brezis \cite[Section 2]{Brezis}. 
For \emph{sufficiently small} $\eps> 0$ and for the disk domain $\Omega=B^2$, Pacard and Rivi\`ere \cite[Theorem 10.2]{Pacard_Riviere} showed  that $E_\eps$ has a unique critical point in $\mcA_2$ and so, it is given by the radially symmetric solution $u_\eps$ in \eqref{def_sol_equa} (for $n=0$). For $N\geq 7$, $\Omega=B^N$ and {\it any} $\eps>0$,  it is proved in  \cite{INSZ_CRAS} that $E_\eps$ has a unique minimizer in $\mcA_N$ which is given by the radially symmetric solution $u_\eps$ in \eqref{def_sol_equa} (for $n=0$). For $2\leq N\leq 6$ and $\Omega=B^N$, Ignat-Nguyen \cite{IN} proved that for any $\eps>0$, $u_\eps$ is a local minimizer of $E_\eps$ in $\mcA$ (which is an extension of the result of Mironescu \cite{Min-stab} in dimension $N=2$). Also, 
Mironescu \cite{Mironescu_symmetry} showed in dimension $N = 2$ that, when $B^2$ is replaced by $\RR^2$ and $\eps = 1$, a local minimizer of $E_\eps$ satisfying a degree-one boundary condition at infinity is unique (up to translation and suitable rotation). This was extended in dimension $N = 3$ by Millot and Pisante \cite{mil-pis} and in dimensions $N \geq 4$ by Pisante \cite{Pisante} in the case of the blow-up limiting problem on $\RR^N$ and $\eps = 1$. All these results (holding for $n=0$) are related to the study of the limit problem obtained by sending $\eps \rightarrow 0$ when the Ginzburg-Landau problem on the unit ball  `converges' to the harmonic map problem from $B^N$ into the unit sphere $\bS^{N-1}$. For that harmonic map problem, the vortex boundary condition yields uniqueness of the minimizing harmonic $\bS^{N-1}$-valued map $x \mapsto \frac{x}{|x|}$ if $N\geq 3$; this is proved by Brezis, Coron and Lieb~\cite{BrezisCoronLieb}  in dimension $N = 3$ and by Lin \cite{Lin} in any dimension $N \geq 3$; we also mention J\"{a}ger and Kaul \cite{JagKaul} in dimension $N \geq 7$ for the equator map $x\in B^N\mapsto (\frac{x}{|x|},0)\in \bS^{N}$.

\noindent {\it II. The case of $n\geq 1$ and $\Omega=B^N\times (0,1)^n$}. As we explain in Remark \ref{import-rem} below, for some $\eps>0$, if the minimality of the radially symmetric solution $u_\eps$ in \eqref{def_sol_equa} holds in the case $n=0$ (so, for $\Omega=B^N$), then this implies the minimality of $u_\eps$ in $\Omega=B^N\times (0,1)^n$ also for every dimension $n\geq 1$. In particular, the result of Pacard-Rivi\`ere \cite[Theorem 10.2]{Pacard_Riviere} for $n=0$ and $N=2$ yields the minimality of $u_\eps$ in \eqref{def_sol_equa} defined in $B^2\times (0,1)^n$ for every $n\geq 1$ if $\eps>0$ is sufficiently small. Also, the result of Ignat-Nguyen-Slastikov-Zarnescu \cite[Theorem 1]{INSZ_CRAS} for $n=0$, $N\geq 7$ and any $\eps>0$ generalizes to dimension $n\geq 1$ for $\Omega=B^N\times (0,1)^n$ (see the proof of Theorem  \ref{GL_equator}). We also mention the work of Sandier-Shafrir \cite{SS2} where they treat the case of topologically trivial $\RR^2$-valued solutions in the domain $\Omega=\RR^3$ (see also \cite{DPMR, S1} for vortex filament solutions). 
 
\subsection{Escaping $\R^M$-valued vortex sheet solutions when $M\geq N+1$}
In dimension $2\leq N\leq 6$ and for $\eps<\eps_N$ given in \eqref{epsN}, a different type of radially symmetric vortex sheet solution appears provided that the target space has dimension $M\geq N+1$. More precisely, we consider the energy functional
$
E_\eps
$ 
in \eqref{en} over the set of $\R^M$-valued maps
\be
\label{mca}
\mcA:=\{ u\in H^1(\Omega; \R^M):\,  u(x,z)=(x, 0_{\R^{M-N}}) \textrm{ on } \partial B^N = \bS^{N-1} \subset \RR^M, z\in (0,1)^n\}.
\ee
If $M\geq N+1$, the prototype of radially symmetric critical points of $E_{\eps}$ in $\mcA$ has the following form (invariant in $z$-direction): \footnote{If $M=N+1$, then $\tilde u_\eps(x,z)=(\tf_{\eps}(r) \frac{x}{|x|}, g_{\eps}(r))$ for every $x\in B^N$ and  $z\in (0,1)^n$. In fact, if $n=0$ (so, for $\Omega=B^N$), every bounded critical point of $E_\eps$ in $\mcA$ that is invariant under the action of a special group (isomorphic to $SO(N)$) has the form of  $\tilde u_\eps$, see \cite[Definition A.1, Lemma A.5]{IN}. }  
\begin{equation}
\tilde u_\eps(x,z)=(\tf_{\eps}(r) \frac{x}{|x|}, 0_{\R^{M-N-1}}, g_{\eps}(r)) \in \mcA, \quad x\in B^N, z\in (0,1)^n, r=|x|,
	\label{Eq:feegeeH1}
\end{equation}
where $(\tf_{\eps},g_{\eps})$ satisfies the system of  ODEs
\begin{align}
-\tf_{\eps}'' - \frac{N-1}{r} \tf_{\eps}' + \frac{N-1}{r^2} \tf_{\eps}
	&= \frac{1}{\eps^2} W'(1 -  \tf_{\eps}^2 - g_{\eps}^2) \tf_{\eps} \quad \textrm{in } (0,1),
	\label{Eq:20III21-fee}\\
-g_{\eps}'' - \frac{N-1}{r} g_{\eps}' 
	&= \frac{1}{\eps^2} W'(1 -  \tf_{\eps}^2 - g_{\eps}^2) g_{\eps}  \quad \textrm{in } (0,1),\label{Eq:20III21-gee}\\
 \tf_{\eps}(1) &= 1  \text{ and } g_{\eps}(1) = 0.
	\label{Eq:20III21-feegeeBC}
\end{align}
We distinguish two type of radial profiles:

\smallskip

$\bullet$ the {\it non-escaping} radial profile $(\tf_\eps=f_\eps, g_\eps=0)$ with the unique radial profile $f_\eps$ given in \eqref{1}; in this case, we say that $\tilde u_\eps=(u_\eps, 0_{\R^{M-N}})$ is a {\it non-escaping} (radially symmetric) vortex sheet solution where $u_\eps$ is given in \eqref{def_sol_equa}.

\smallskip

$\bullet$ the {\it escaping} radial profile $(\tf_\eps, g_\eps)$ with $g_\eps>0$ in $(0,1)$; in this case, we call an {\it escaping} (radially symmetric) vortex sheet solution $\tilde u_\eps$ in \eqref{Eq:feegeeH1}. In this case, $\tf_\eps\neq f_\eps$ and obviously, $(\tf_\eps, -g_\eps)$ is another radial profile to \eqref{Eq:feegeeH1}-\eqref{Eq:20III21-feegeeBC}.

\smallskip

The properties of such radial profiles (e.g., existence, uniqueness, minimality, monotonicity) are analyzed in Theorem \ref{Thm:ExtendedExist} below and are based on ideas developed by Ignat-Nguyen \cite{IN}. 

Our main result proves the radial symmetry of global minimizers of $E_\eps$ in $\mcA$. More precisely, the
following dichotomy occurs at $\eps_N$ defined in \eqref{epsN}: if $\eps<\eps_N$, then {\it escaping} radially symmetric vortex sheet solutions exist and determine (up to certain orthogonal transformations) the full set of global minimizers of $E_\eps$ in $\mcA$; if instead $\eps\geq \eps_N$, then the {\it non-escaping} radially symmetric vortex sheet solution is the unique global minimizer of $E_{\eps}$ in $\mcA$ and no escaping radially symmetric vortex sheet solutions exist in this case.

\begin{theorem}\label{thm:dicho}
Let $n\geq 1$, $2 \leq N \leq 6$, $M\geq N+1$, $W \in C^2((-\infty,1])$ satisfy \eqref{Eq:26VI18-E1} and be strictly convex.
Consider $\eps_N \in (0,\infty)$ such that $\ell(\eps_N)=0$ in \eqref{epsN}.
Then there exists an escaping radially symmetric vortex sheet solution $\tilde u_\eps$ in \eqref{Eq:feegeeH1} with $g_\eps>0$ in $(0,1)$ if and only if $0 < \eps < \eps_N$. Moreover,   

\begin{enumerate}

\item if $0 < \eps < \eps_N$, the escaping radially symmetric vortex sheet solution $\tilde u_\eps$ is a global minimizer of $E_\eps$ in $\mcA$ and all global minimizers of  $E_\eps$ in $\mcA$ are radially symmetric given by $R\tilde u_\eps$ where $R\in O(M)$ is an orthogonal transformation of $\R^M$ satisfying $Rp=p$ for all $p\in \R^N\times \{0_{\R^{M-N}}\}$. In this case, the non-escaping vortex sheet solution $(u_\eps, 0_{\R^{M-N}})$ in \eqref{def_sol_equa} is an unstable critical point of $E_{\eps}$ in $\mcA$. 

\item if $\eps\geq \eps_N$, the non-escaping vortex sheet solution $(u_\eps, 0_{\R^{M-N}})$ in \eqref{def_sol_equa} is the unique global minimizer of $E_\eps$ in $\mcA$. Furthermore, there are no bounded critical points $w_\eps$ of $E_{\eps}$ in $\mcA$ that escape in some direction $e\in \bS^{M-1}$ (i.e., $w_\eps\cdot e>0$ a.e. in $\Omega$).

\end{enumerate}
\end{theorem}

The result above holds also if $n=0$, i.e., $\Omega=B^N$ and the vortex sheets corresponding to  the above solutions become vortex points (see Theorem \ref{thm:ball}). It generalizes  \cite[Theorem 1.1]{INSZ_ENS} that was proved in the case $N=2$ and $M=3$ (without identifying the meaning of the dichotomy parameter $\eps_N$ in \eqref{epsN}). The dichotomy in 
Theorem~\ref{thm:dicho} happens in dimensions $2\leq N\leq 6$ because of the phenomenology occurring for the limit problem $\eps\to 0$. More precisely, if $M\geq N+1$, then minimizing $\mathbb{S}^{M-1}$-valued harmonic maps in $\mcA$ are smooth and escaping in a direction of $\bS^{M-1}$ provided that $N\leq 6$; if $N\geq 7$, then there is a unique minimizing $\mathbb{S}^{M-1}$-valued harmonic maps in $\mcA$, non-escaping and singular, the singular set  being given by a vortex sheet of dimension $n$ in $\Omega$ (see Theorem \ref{thm:dico-HMP} in Appendix below). 
This suggests why in dimension $N\geq 7$ and for any $\eps>0$, there is no escaping radially symmetric vortex sheet critical point $\tilde u_\eps$ of  $E_\eps$ in $\mcA$ while the non-escaping vortex sheet solution $(u_\eps, 0_{\R^{M-N}})$ is the unique global minimizer of $E_\eps$ in $\mcA$ (see Theorem \ref{GL_equatorNM} and Remark \ref{rem:dichoto} below).

The paper is meant to be self-contained and it is organized as follows. In Section \ref{sec-2}, we prove the minimality and the uniqueness results for the non-escaping radially symmetric solution in Theorems  \ref{GL_equator} and \ref{eps_large}; this is done in a more general setting by considering the target dimension $M\geq N$ for the set of configurations $\mcA$ instead of $\mcA_N$. Section \ref{sec-3} is devoted to characterize escaping vortex sheet solutions. First, we prove the minimality of such bounded solutions stated in Theorem \ref{thm:main}. Second, we prove existence, minimality and uniqueness results for the escaping radial profile in Theorem \ref{Thm:ExtendedExist}. Finally, we prove our main result on the dichotomy between escaping / non-escaping radially symmetric vortex sheet solutions in Theorem \ref{thm:dicho}. In Appendix, we prove the corresponding dichotomy result for $\bS^{M-1}$-valued harmonic maps in Theorem \ref{thm:dico-HMP} which again is based on the minimality of escaping $\bS^{M-1}$-valued harmonic maps in Theorem \ref{prop:HMP}.

\bigskip

\noindent{\bf Acknowledgment.} R.I. is partially supported by the ANR projects ANR-21-CE40-0004 and ANR-22-CE40-0006-01. He also thanks for the hospitality of the Hausdorff Research Institute for Mathematics in Bonn during the trimester ``Mathematics for Complex Materials".

\section{The non-escaping vortex sheet solution. Proof of Theorems \ref{GL_equator} and \ref{eps_large}}
\label{sec-2}

Theorem \ref{GL_equator} will be obtained as a consequence of a stronger result on the uniqueness of global minimizers of  the $\RR^M$-valued Ginzburg-Landau functional with $M \geq N\geq 7$. For that, we consider the energy functional
$
E_\eps
$ 
in \eqref{en} over the set $\mcA$ defined in \eqref{mca}. 
The aim is to prove the minimality and uniqueness of the vortex sheet solution $(u_\eps, 0_{\R^{M-N}})$ where $u_\eps$ given in \eqref{def_sol_equa} with the obvious identification $u_\eps\equiv (u_\eps, 0_{\R^{M-N}})$ if $M=N$, following the ideas of Ignat-Nguyen-Slastikov-Zarnescu \cite{INSZ_ENS, INSZ_CRAS}.

\begin{theorem}
\label{GL_equatorNM}
Assume that $W$ satisfies \eqref{Eq:26VI18-E1} and $n\geq 1$. If $M \geq N\geq 7$, then for every $\eps>0$, $(u_\eps, 0_{\R^{M-N}})$ given in \eqref{def_sol_equa} is the unique global minimizer of $E_\eps$  in $\mcA$.
\end{theorem}

\begin{proof}
To simplify notation, we identify 
\be
\label{identif}
u_\eps\equiv (u_\eps, 0_{\R^{M-N}}) \quad \textrm{ when } \quad  M\geq N.
\ee
The proof will be done in several steps following the strategy in \cite[Theorem 1.7]{INSZ_ENS}, \cite[Theorem 1]{INSZ_CRAS}. First, for an arbitrary competitor $u_\eps+v$, we consider the excess energy $E_\eps(u_\eps+v)-E_\eps(u_\eps)$ for the critical point $u_\eps$ defined in \eqref{def_sol_equa} and show a lower estimate by a quadratic energy functional $F_\eps(v)$ coming from the operator $L_\eps$ in \eqref{GL-op}. Second, we show that $F_\eps(v)\geq 0$ using the properties of the radial profile $f_\eps$ in \eqref{1} and a Hardy decomposition method; this proves in particular that $u_\eps$ is a global minimizer of $E_\eps$ over $\mcA$. Finally, by analyzing the zero excess energy states, we conclude to the uniqueness of the global minimizer $u_\eps$. 

\bigskip
\par\noindent{\it Step 1: Excess energy.} 
For any  $v\in H^1_0(B^N\times \R^n; \R^M)$, we have
\begin{align*}
E_\eps(u_\eps+v)-E_\eps(u_\eps)
	&=\int_{\Omega}\Big[ \nabla u_\eps\cdot\nabla v +\frac{1}{2}|\nabla v|^2\Big]\,dx dz\\
		&\qquad +\frac{1}{2\eps^2}\int_{\Omega} \Big[W(1-|u_\eps+v|^2)-W(1-|u_\eps|^2)\Big]\,dx dz.
\end{align*}
Note that for every $u\in \mcA$, $u_\eps-u$ can be extended to $v\in H^1_0(B^N\times \R^n; \R^M)$. In particular, $v(\cdot, z)\in H^1_0(B^N, \R^M)$ for a.e. $z\in (0,1)^n$.  
The convexity of $W$ yields
\be
\label{111}
W(1-|u_\eps+v|^2)-W(1-|u_\eps|^2)\ge -W'(1-|u_\eps|^2)(|u_\eps+v|^2-|u_\eps|^2).
\ee
Combining the above relations, we obtain the following lower bound for the excess energy:
\begin{align}
\nonumber
E_\eps(u_\eps+v)-E_\eps(u_\eps)
	&\ge \int_{\Omega} \Big[\nabla u_\eps\cdot\nabla v-\frac{1}{\eps^2} W'(1-f_\eps^2)u_\eps\cdot v\Big]\,dxdz\nonumber\\\nonumber
		&\qquad\qquad +\int_{\Omega} \Big[ \frac{1}{2} |\nabla v|^2-\frac{1}{2\eps^2}W'(1-f_\eps^2)|v|^2\Big]dx dz\\
\label{rel:endif3} &=\int_{\Omega} \frac 12 |\nabla_z v|^2\, dx dz+\int_{(0,1)^n}\frac12 F_\eps(v(\cdot, z))\, dz,
\end{align}
where we used the PDE \eqref{E-L} and introduced the quadratic functional $$F_\eps(\Psi)=\int_{B^N} \Big[  |\nabla_x \Psi|^2-\frac{1}{\eps^2}W'(1-f_\eps^2)|\Psi|^2\Big]dx,$$
for all  $\Psi\in H^1_0(B^N; \R^M)$.
Note that the $L^2$-gradient of $F_\eps$ represents a part of the linearization of the PDE \eqref{E-L} at $u_\eps$ and it is given by the operator $L_\eps$  in \eqref{GL-op}. The rest of the proof is devoted to show that for $N\geq 3$:
$$F_\eps(\psi)\geq \bigg(\frac{(N-2)^2}{4}-(N-1)\bigg)\int_{{B^N}}\frac{\psi^2}{r^2}\, dx, \quad \forall \psi\in H^1_0(B^N)$$
yielding the conclusion for $N\geq 7$ and also the inequality for the first eigenvalue $\ell(\eps)$ of the operator
$
L_\eps$ in \eqref{GL-op} in $B^N$: \footnote{Observe the difference between dimension $N\geq 7$ and the case of dimension $2\leq N\leq 6$ where we have $\ell(\eps)<0$ for $\eps<\eps_N$ in \eqref{epsN}; moreover, if $N\leq 6$, then $\ell(\eps)$ blows up as $-\frac1{\eps^2}$ as $\eps\to 0$ (see \cite[Lemma 2.3]{IN}).}
$$\ell(\eps)\geq \frac{(N-2)^2}{4}-(N-1)>0, \quad \forall \eps>0 \quad \textrm{and} \quad N\geq 7.$$ To keep the paper self-contained, we explain in the following the simple idea used in \cite{INSZ_ENS, INSZ_CRAS}.

\bigskip
\par\noindent{\it Step 2: A factorization argument}. As $f_\eps>0$ is a smooth positive radial profile in $(0,1)$, we decompose every scalar test function $\psi\in C^\infty_c(B^N\setminus \{0\}; \R)$ as follows
$$\psi(x)=f_\eps(r) w(x), \quad \forall x\in B^N\setminus \{0\}, \, r=|x|,$$  
where $w\in C^\infty_c(B^N\setminus \{0\}; \R)$. Integrating by parts (see e.g. \cite[Lemma~A.1]{INSZ3}), we deduce:
\begin{align*}
F_\eps(\psi)=\int_{B^N} L_\eps \psi\cdot \psi\, dx&=\int_{B^N} w^2 (L_\eps f_\eps\cdot f_\eps)\, dx+\int_{B^N} f_\eps^2 |\nabla_x w|^2\, dx\\
& =\int_{B^N} f_\eps^2 \bigg( |\nabla_x w|^2-\frac{N-1}{r^2}w^2\bigg)\, dx,
\end{align*}
because $L_\eps f_\eps\cdot f_\eps=-\frac{N-1}{r^2} f_\eps^2$ in ${B^N}$ by \eqref{1}. 
Furthermore, we decompose $$w=\f g \quad \textrm{in} \quad B^N\setminus \{0\}$$ with $\f=|x|^{-\frac{N-2}{2}}$ satisfying 
$$-\Delta_x \f=\frac{(N-2)^2}{4|x|^2}\f \quad \textrm{ in } \, \R^N\setminus\{0\}$$ and $g\in C^\infty_c({B^N}\setminus \{0\}; \R)$. Then
\begin{align*}
|\nabla_x w|^2=|\nabla_x g|^2 \f^2+|\nabla_x \f|^2 g^2+\frac12 \nabla_x (\f^2)\cdot \nabla_x (g^2).
\end{align*}
As $|\nabla_x \f|^2=\frac{(N-2)^2}{4|x|^2}\f^2$ and $\f^2$ is harmonic in ${B^N}\setminus\{0\}$ (recall that $N\geq 7$), integration by parts yields
\begin{align}
\nonumber
F_\eps(\psi)&=\int_{B^N} f_\eps^2 \bigg( |\nabla_x g|^2 \f^2+\frac{(N-2)^2}{4r^2}\f^2g^2-\frac{N-1}{r^2}\f^2 g^2\bigg)\, dx-\frac12 \int_{{B^N}} \nabla_x (\f^2)\cdot \nabla_x (f_\eps^2)g^2\, dx\\
\nonumber &\geq \int_{B^N} f_\eps^2 |\nabla_x g|^2 \f^2\, dx+\bigg(\frac{(N-2)^2}{4}-(N-1)\bigg)\int_{{B^N}}\frac{f_\eps^2}{r^2}\f^2g^2\, dx\\
\label{23}
&\geq \bigg(\frac{(N-2)^2}{4}-(N-1)\bigg)\int_{{B^N}}\frac{\psi^2}{r^2}\, dx\geq 0,
\end{align}
where we used $N\geq 7$ and $\frac12\nabla_x (\f^2)\cdot \nabla_x (f_\eps^2)=2\f \f' f_\eps f'_\eps\leq 0$  in ${B^N}\setminus \{0\}$ because $\f, f_\eps, f'_\eps>0$ and $\f'<0$ in $(0,1)$ (see e.g. \cite{HH, ODE_INSZ, IN}).

\bigskip
\par\noindent{\it Step 3: We prove that $F_\eps(\Psi)\geq 0$ for every $\Psi\in H^1_0(B^N; \R^M)$; moreover, $F_\eps(\Psi)=0$ if and only if $\Psi=0$.} Let $\Psi\in H^1_0(B^N; \R^M)$. As a point in $\R^N$ has zero $H^1$ capacity, a standard density argument implies the existence of a sequence $\Psi_k\in C^\infty_c({B^N}\setminus \{0\}; \R^M)$ such that $\Psi_k\to \Psi$ in $H^1(B^N, \R^M)$ and a.e. in $B^N$. On the one hand, by 
definition of $F_\eps$, since $W'(1-f_\eps^2)\in L^\infty$, we deduce that $F_\eps(\Psi_k)\to F_\eps(\Psi)$ as $k\to \infty$.
On the other hand, by \eqref{23} and Fatou's lemma, we deduce
\begin{align*}
\liminf_{k\to \infty} F_\eps(\Psi_k)&\geq  \bigg(\frac{(N-2)^2}{4}-(N-1)\bigg)\liminf_{k\to \infty} \int_{{B^N}}\frac{|\Psi_k|^2}{r^2}\, dx\\
&\geq \bigg(\frac{(N-2)^2}{4}-(N-1)\bigg) \int_{{B^N}}\frac{|\Psi|^2}{r^2}\, dx.
\end{align*}
Therefore, we conclude that
$$F_\eps(\Psi)\geq \bigg(\frac{(N-2)^2}{4}-(N-1)\bigg) \int_{{B^N}}\frac{|\Psi|^2}{r^2}\, dx\geq 0,\quad \forall \Psi\in H^1_0(B^N; \R^M).$$
Moreover, $F_\eps(\Psi)=0$ if and only if $\Psi=0$.

\bigskip
\par\noindent{\it Step 4: Conclusion.}  By \eqref{rel:endif3} and Step 3,  we deduce that $u_\eps$ is a global minimizer of 
$E_\eps$ over $\mcA$. For uniqueness, assume that $\hat u_\eps$ is  another global minimizer of $E_\eps$ over $\mcA$. If $v:=\hat u_\eps-u_\eps$,
then $v$ can be extended in $H^1_0(B^N\times \R^n; \R^M)$ and by Steps 1 and 3, we have that 
$$0=E_\eps(\hat u_\eps)-E_\eps(u_\eps)\geq \int_{\Omega} \frac 12 |\nabla_z v|^2\, dx dz+\int_{(0,1)^n}\frac12 
F_\eps(v(\cdot, z))\, dz\geq 0,$$ which yields $\nabla_z v=0$ a.e. in $\Omega$ and $F_\eps(v(\cdot, z))=0$ for a.e. $z\in (0,1)^n$. In other words, $v=v(x)$ and Step 3 implies that $v=0$, i.e., $\hat u_\eps=u_\eps$ in $\Omega$.
 \end{proof}

\begin{remark}
\label{import-rem}
Theorem \ref{GL_equatorNM} reveals the following fact: if for $n=0$ (i.e., $\Omega=B^N$) and some $\eps>0$, a (radially symmetric) critical point $\hat u_\eps:B^N\to \RR^M$ of $E_\eps$ in $\mcA$ is proved to be a global minimizer (and additionally, if one proves that it is the unique global minimizer),  then for any dimensions $n\geq 1$ (i.e., $\Omega=B^N\times (0,1)^n$), this $z$-invariant solution $\hat u_\eps$ of \eqref{E-L} in $B^N\times (0,1)^n$ is also a global minimizer (and additionally, it is the unique minimizer) of $E_\eps$ in $\mcA$. This is because for every $u:B^N\times (0,1)^n\to \RR^M$ with $u\in \mcA$, then $u(\cdot, z)$ satisfies the degree-one vortex boundary condition on $\partial B^N$ for every $z\in (0,1)^n$ yielding
\begin{align*}
E_\eps(u)&=\int_{\Omega} \frac 12 |\nabla_z u|^2\, dx dz+\int_{(0,1)^n}E_\eps(u(\cdot, z))\, dz\\
&\geq \int_{(0,1)^n}E_\eps(\hat u_\eps)\, dz=E_\eps(\hat u_\eps);
\end{align*}
the equality occurs only when $u$ is $z$-invariant. Thus,  if the uniqueness of the global minimizer $\hat u_\eps$ holds in $B^N$ (i.e., $n=0$), then this yields uniqueness of the global minimizer $\hat u_\eps$ in $\Omega=B^N\times (0,1)^n$ (as a map independent of $z$-variable) for every $n\geq 1$.
\end{remark}

\begin{proof}[Proof of Theorem \ref{eps_large}] We prove the result in the more general setting of $\R^M$-valued maps $u$ belonging to $\mcA$ for $M\geq N$ using the same identification \eqref{identif}. By Step 1 in the proof of Theorem \ref{GL_equatorNM} (see \eqref{rel:endif3}), the excess energy is estimated for every  $v\in H^1_0(B^N\times \R^n; \R^M)$:
\begin{align*}
E_\eps(u_\eps+v)-E_\eps(u_\eps) &\ge \int_{\Omega} \frac 12 |\nabla_z v|^2\, dx dz+\frac12 \int_{(0,1)^n} <L_\eps v(\cdot, z), v(\cdot, z) >\, dz,
\end{align*}
where $L_\eps$ is the operator in \eqref{GL-op} and $<\cdot,\cdot>$ denotes the duality pairing $(H^{-1},H^1_0)$ in $B^N$. If $\eps\geq \eps_N$, then $\ell(\eps)\geq 0$ (by \cite[Lemma 2.3]{IN}) and therefore, \footnote{\label{foot} Indeed, for a scalar function $v\in C^\infty_c(B^N\setminus \{0\}, \R)$, if  
$\psi=\psi(r)>0$ is a radial first eigenfunction of $L_\eps$ in $B^N$ with zero Dirichlet data, i.e., $L_\eps \psi=\ell(\eps)\psi$ in $B^N$, then the duality pairing $(H^{-1},H^1_0)$ term in $B^N$ writes (see e.g. \cite[Lemma~A.1]{INSZ3}):
$$
<L_\eps v, v>\, \, =\int_{B^N} \psi^2 |\nabla (\frac{v}\psi)|^2\, dx+\int_{B^N} (\frac{v}\psi)^2 L_\eps \psi \cdot \psi\, dx=\int_{B^N} \psi^2 |\nabla (\frac{v}\psi)|^2\, dx+\ell(\eps) \|v\|^2_{L^2(B^N)}.$$ By a density argument, Fatou's lemma yields for every scalar function $v\in H^1_0(B^N, \R)$, $$<L_\eps v, v>\, \, \geq \int_{B^N} \psi^2 |\nabla (\frac{v}\psi)|^2\, dx+\ell(\eps) \|v\|^2_{L^2(B^N)}.$$ } 
\be
\label{spec}
<L_\eps v(\cdot, z), v(\cdot, z) >\, \, \geq \ell(\eps) \|v(\cdot, z)\|^2_{L^2(B^N)}\geq 0 \quad \textrm{ for a.e. } z\in (0,1)^n,
\ee
where we used that $v(\cdot, z)\in H^1_0(B^N; \R^M)$ for a.e. $z\in (0,1)^n$. 
Thus, $u_\eps$ is a minimizer of $E_\eps$ over $\mcA$. It remains to prove uniqueness of the global minimizer. For that, if $\hat u_\eps$ is  another global minimizer of $E_\eps$ over $\mcA$, setting $v:=\hat u_\eps-u_\eps$,
then $v$ can be extended in $H^1_0(B^N\times \R^n; \R^M)$ and 
\be
\label{egal}
0=E_\eps(\hat u_\eps)-E_\eps(u_\eps)\geq \int_{\Omega} \frac 12 |\nabla_z v|^2\, dx dz+\frac{\ell(\eps)}2
\int_{(0,1)^n} 
\int_{B^N} |v(x, z)|^2\, dx dz\geq 0
\ee
because $\ell(\eps)\geq 0$ for $\eps\geq \eps_N$. Thus, equality holds in the above inequalities.

\bigskip
\par\noindent{\it Case 1: $\eps>\eps_N$}. In this case, $\ell(\eps)>0$ and we conclude that $v=0$ in $\Omega$, i.e., $\hat u_\eps=u_\eps$ in $\Omega$.

\bigskip
\par\noindent{\it Case 2: $\eps=\eps_N$ and $W$ is in addition strictly convex}. In this case, $\ell(\eps)=0$ and by \eqref{egal}, $v$ is invariant in $z$, i.e., $v=v(x)$ and equality holds in \eqref{spec} and in \eqref{rel:endif3}, thus, equality holds in \eqref{111}. Note that by footnote \ref{foot} the equality in \eqref{spec} holds if and only if $v=\lambda \psi$ for some $\lambda \in \R^M$, where $\psi=\psi(r)$ is a radial first eigenfunction of $L_\eps$ in $B^N$ with zero Dirichlet data, 
in particular $\psi>0$ in $[0,1)$ and $\psi(1)=0$. Also, by the strict convexity of $W$, the equality \eqref{111} is achieved if and only if $|u_\eps+v|=|u_\eps|$ a.e. in $\Omega$, that is, $|v|^2+2v\cdot u_\eps=0$ a.e. in $B^N$. It yields
\be
\label{egal12}
|\lambda|^2\psi^2+2f_\eps(|x|)(\frac{x}{|x|}, 0_{\R^{M-N}})\cdot \lambda \psi=0 \quad \textrm{for every } x\in B^N.\ee 
Dividing by $\psi$ in $B^N$, the continuity up to the boundary $\partial B^N$ leads to $2f_\eps(|x|)(x, 0_{\R^{M-N}})\cdot \lambda=0$ for every $x\in \partial B^N$ since $\psi=0$ on $\partial B^N$. As $f_\eps(1)=1$, it follows that the first $N$ components of $\lambda$ vanish. Coming back to \eqref{egal12}, we conclude that $|\lambda|^2\psi^2=0$ in $B^N$, i.e., $\lambda=0$ and so, $v=0$ and $\hat u_\eps=u_\eps$ in $\Omega$.
\end{proof}

\section{Properties of escaping vortex sheet solutions when $M\geq N+1$}
\label{sec-3}

\subsection{Minimality of escaping vortex sheet solutions}

In this section, we require the additional assumption of strict convexity of $W$ in order to determine the set of global minimizers of $E_\eps$ over $\mcA$ in \eqref{mca}. However, $W$ is assumed to be only $C^1$ not $C^2$. We prove that every bounded solution to \eqref{E-L} escaping in some direction is a global minimizer of $E_\eps$ over $\mcA$; moreover, such global minimizer is unique up to an orthogonal transformation of $\R^M$ keeping invariant the space $\R^N\times \{0_{\R^{M-N}}\}$.

\begin{theorem}\label{thm:main}
We consider the dimensions $n\geq 1$ and $M> N \geq 2$, the potential $W\in C^1((-\infty,1],\R)$ satisfying \eqref{Eq:26VI18-E1} and an escaping direction $e\in \bS^{M-1}$. Fix any $\eps>0$ and let $w_\eps \in H^1\cap L^\infty(\Omega,\R^M)$ be a critical point  of the energy $E_\eps$ in the set $\mcA$ which is positive in the direction $e$ inside $\Omega$:
\be\label{ass:Phi}
w_\eps \cdot {e}>0\textrm{ a.e. in }\Omega.
\ee 
Then $w_\eps$ is a global minimizer of $E_\eps$ in $\mcA$. If in addition $W$ is strictly convex, then all minimizers of  $E_\eps$ in $\mcA$ are given by $Rw_\eps$ where $R\in O(M)$ is an orthogonal transformation of $\R^M$ satisfying $Rp=p$ for all $p\in \R^N\times \{0_{\R^{M-N}}\}$.
\end{theorem}

This result is reminiscent from \cite[Theorem 1.3]{INSZ_ENS}. However, it doesn't apply directly as the domain $\Omega$ is not smooth here and the boundary condition is a mixed Dirichlet-Neumann condition (w.r.t. Dirichlet boundary condition in \cite{INSZ_ENS}). 

\begin{proof}
In the following, we denote the variable $X=(x,z)\in \Omega=B^N\times (0,1)^n$. As a critical point of $E_\eps$ in the set $\mcA$, $w_\eps:\Omega \to \R^M$ satisfies 
\be
\label{E-L-M}
\left\{\begin{array}{l}
-\Delta w_\eps=\frac1{\eps^2} w_\eps \, W'(1-|w_\eps|^2)\quad \textrm{ in } \, \Omega,\\
\frac{\partial w_\eps}{\partial z}=0 \quad \textrm{ on } \, B^N\times \partial (0,1)^n,\\
w_\eps(x,z)=(x, 0_{\R^{M-N}}) \quad \textrm{ on } \, \partial B^N\times (0,1)^n.
\end{array}
\right.
\ee
In particular, $\Delta w_\eps \in L^\infty(\Omega)$ (as $W'$ is continuous and $w_\eps\in L^\infty(\Omega)$); then standard elliptic regularity for the mixed boundary conditions in \eqref{E-L-M} yields $w_\eps\in C^1(\bar \Omega, \R^M)$. Thus, \eqref{ass:Phi} implies $w_\eps\cdot e\geq 0$ in $\bar \Omega$ and the vortex boundary condition in $\mcA$ implies that $e$ is orthogonal to $\R^N\times \{0_{\R^{M-N}}\}$.
By the invariance of the energy and the vortex boundary condition under the transformation $w_\eps(X)\mapsto Rw_\eps(X)$ for any $R\in O(M)$ satisfying $Rp=p$ for all $p\in \R^N\times \{0_{\R^{M-N}}\}$, we know that $Rw_\eps$ is also a critical point of $E_\eps$ over $\mcA$; thus, we can assume that 
\be
\label{def-enu}
e:=e_M=(0,\dots, 0, 1)\in \R^M.
\ee
We prove the result in several steps.

\bigskip
\par\noindent{\it Step 1: Excess energy.} 
By Step 1 in the proof of Theorem \ref{GL_equatorNM}, we have for any  $v\in H^1_0(B^N\times \R^n, \R^M)$: 
\be\label{rel:endif4}
E_\eps(w_\eps+v)-E_\eps(w_\eps) \ge  \int_\Omega \Big[ \frac 12 |\nabla v|^2-\frac{1}{2\eps^2} W'(1-|w_\eps|^2) |v|^2\Big]\,dX =: \frac12G_\eps(v)
\ee
(note that $G_\eps(v)$ is larger than the integration of $F_\eps(v)$ in \eqref{rel:endif3} over $(0,1)^n$ as it contains also the integration of $|\nabla_z v|^2$). 
If in addition $W$ is strictly convex, then equality holds above if and only if $|w_\eps(X)+v(X)|=|w_\eps(X)|$ a.e. $X\in\Omega$ (by \eqref{111}).

\bigskip
\par\noindent{\it Step 2: Global minimality of $w_\eps$.} It is enough to show that the quadratic energy $G_\eps(v)$ defined in \eqref{rel:endif4} is nonnegative for any  $v\in H^1_0(B^N\times \R^n, \R^M)$. Denoting the $M$-component of $w_\eps$ by $\phi:=w_\eps \cdot e_M$, we know that $\phi\in C^1(\bar \Omega)$, $\phi\geq 0$ in $\Omega$ (by \eqref{ass:Phi}) and satisfies the Euler-Lagrange equation in the sense of distributions:
\be\label{eq:phi}
\left\{\begin{array}{l}
-\Delta \phi -\frac{1}{\eps^2}W'(1-|w_\eps|^2) \phi =0 \, \,  \textrm{in } \, \Omega,\\ 
\phi=0  \, \,  \textrm{ on } \, \partial B^N\times (0,1)^n,\\ 
\frac{\partial \phi}{\partial z}=0 \, \,  \textrm{ on } \, B^N\times \partial (0,1)^n.
\end{array}
\right.
\ee
Note that by strong maximum principle, $\phi>0$ in $\Omega$ (as $\phi$ cannot be identically $0$ in $\Omega$ by \eqref{ass:Phi}). Moreover, Hopf's lemma yields $\phi>0$ on $B^N\times \partial (0,1)^n$ as $\frac{\partial \phi}{\partial z}$ vanishes there. Now, for any smooth map $v \in C_c^\infty (B^N\times \R^{n}; \R^M)$, we can define $\Psi= \frac{v}{\phi} \in C^1(\bar \Omega; \R^M)$ with $\Psi=0$ in a neighborhood of $\partial B^N\times (0,1)^n$ and integration by parts yields for every component $v_j=\phi \Psi_j$ with $1 \leq j \leq M$ (as in \cite[Lemma~A.1.]{INSZ3}):
\begin{align*}
G_\eps(v_j)&=\int_\Omega \Big[|\nabla v_j|^2-\frac{1}{\eps^2}W'(1-|w_\eps|^2) \phi \cdot \phi \Psi_j^2\Big]\,dX\\
&\stackrel{\eqref{eq:phi}}{=} \int_\Omega \Big[|\nabla (\phi \Psi_j)|^2-\nabla \phi \cdot \nabla (\phi\,\Psi_j^2) \Big]\,dX
= \int_\Omega \phi^2 |\nabla \Psi_j|^2\, dX.
\end{align*}
As $G_\eps$ is continuous in strong $H^1(\Omega)$ topology (since $W'(1-|w_\eps|^2)\in L^\infty(\Omega)$), by density of $C_c^\infty (B^N\times \R^n; \R^M)$ in $H^1_0(B^N \times \R^n; \R^M)$, Fatou's lemma yields 
$$G_\eps(v)\geq \int_\Omega \phi^2 |\nabla \big(\frac{v}{\phi} \big)|^2\,dX\ge 0, \quad \forall v \in H^1_0(B^N \times \R^n; \R^M).$$ 
As a consequence of \eqref{rel:endif4}, we deduce that $w_\eps$ is a minimizer of $E_\eps$ over $\mcA$.
Moreover, $G_\eps(v)=0$ if and only if there exists a (constant) vector $\lambda \in \RR^M$ such that $v = \lambda \phi$ for a.e. $x\in\Omega$.

\bigskip
\par\noindent{\it Step 3: Set of global minimizers.} From now on, we assume that $W$ is strictly convex and denote $w_\eps=(w_{\eps,1}, \dots, w_{\eps,M})$. Note that the map
\be
\label{transf_n}
\tilde w_\eps:=(w_{\eps,1},\dots, w_{\eps,N}, 0_{\R^{M-N-1}}, \sqrt{w^2_{\eps, N+1}+\dots+w_{\eps, M}^2}) 
\ee
belongs to $\mcA$, $|\tilde w_\eps|=|w_\eps|$ and $|\nabla \tilde w_\eps|\leq |\nabla w_\eps|$ in $\Omega$, so $E_\eps(w_\eps)\geq E_\eps(\tilde w_\eps)$ and 
$$\sqrt{w^2_{\eps, N+1}+\dots+w_{\eps, M}^2}\geq w_{\eps, M}=\phi>0\quad \textrm{in} \quad \Omega.$$ Hence, $\tilde w_\eps$ is a minimizer of $E_\eps$ on $\mcA$ (as $w_\eps$ minimizes $E_\eps$ over  $\mcA$ by Step 2). Therefore, up to interchanging $w_\eps$ and $\tilde w_\eps$, we may assume 
$$
\left\{\begin{array}{l}
w_{\eps, N+1}=\dots=w_{\eps, M-1}\equiv 0 \textrm{ in }\Omega\\
w_{\eps, M}=\phi\stackrel{\eqref{ass:Phi}}{>}0\textrm{ in }\Omega.
\end{array}\right.
$$
We now consider another minimizer $U_\eps$ of $E_\eps$ over $\mcA$ and denote  $v:=U_\eps-w_\eps\in H^1_0(B^N \times \R^n; \R^M)$ after a suitable extension. 
 From Steps 1 and 2 we know that $E_\eps(U_\eps)=E_\eps(v+w_\eps)=E_\eps(w_\eps)$, $G_\eps(v)=0$, $|v+w_\eps|=|w_\eps|$ a.e. in $\Omega$ 
 and $v=\lambda \phi$ for some $\lambda=(\lambda_1, \dots, \lambda_M) \in \R^M$ where we recall that $\phi=w_\eps \cdot e_M$.
By continuity of $w_\eps$ and $\phi$, the relation $|v+ w_\eps|=|w_\eps|$ a.e. in $\Omega$ implies $2 w_\eps\cdot v+|v|^2=0$ everywhere in $\Omega$. Since $v=\lambda\phi$, dividing by $\phi>0$ in $\Omega$, we obtain 
\be
\label{numarul}
2\lambda\cdot  w_\eps+\phi |\lambda|^2=0 \hbox{ in } \Omega
\ee
and by continuity, the equality holds also on $\partial \Omega$. 
As for every $(x,z)\in \partial B^N\times (0,1)^n$, $\phi(x,z)=0$ and $w_\eps(x, z)=(x, 0_{\R^{M-N}})$, we deduce that 
$\lambda\cdot (x, 0_{\R^{M-N}})=0$ for every $x\in \partial B^N$. It follows that $\lambda_1=\lambda_2=\dots=\lambda_N=0$ and therefore, recalling that $w_{\eps, N+1}=\dots=w_{\eps, M-1}= 0 \textrm{ in }\Omega$, we have
by \eqref{numarul}:
\[
2\lambda_M\phi+(\lambda_{N+1}^2+\dots+\lambda_M^2)\phi=0\textrm{ in }\Omega.
\]
As $\phi>0$ in $\Omega$, we obtain
\[
\lambda_{N+1}^2+\dots+\lambda_{M-1}^2+(\lambda_{M} +1)^2=1;
\]
hence we can find $R\in O(M)$ such that $Rp=p$ for all $p\in \R^N\times \{0_{\R^{M-N}}\}$ and $$Re_M=(0,\dots,0,\lambda_{N+1},\dots,\lambda_{M-1}, \lambda_M+1).$$ This implies $U_\eps=w_\eps+v=w_\eps+\lambda \phi=Rw_\eps$ as required. The converse statement is obvious: if $w_\eps$ is a minimizer of $E_\eps$ over $\mcA$ and $R\in O(M)$ is a transformation fixing all points of $\R^N\times \{0_{\R^{M-N}}\}$, then $Rw_\eps$ is also a minimizer of $E_\eps$ over $\mcA$ (because $E_\eps$ and the boundary condition in $\mcA$ are invariant under such orthogonal transformation $R$). 
\end{proof}

\begin{remark}
\label{rem:dichoto}
Note that if $n\geq 1$, $M>N\geq 7$ and $W$ satisfies \eqref{Eq:26VI18-E1} (not necessarily strictly convex), then there are no bounded critical points  of the energy $E_\eps$ in the set $\mcA$ escaping in a direction $e\in \mathbb{S}^{M-1}$. Indeed, if such an escaping critical point of $E_\eps$ in $\mcA$ exists, then by Theorem \ref{thm:main}, this solution would be a global minimizer of $E_\eps$ in $\mcA$ which is a contradiction with the uniqueness of the global minimizer $(u_\eps, 0_{\R^{M-N}})$ in \eqref{def_sol_equa} (that is non-escaping) proved in Theorem \ref{GL_equatorNM}.
 \end{remark}

\subsection{Escaping radial profile}

Let $M\geq N+1$. We give a necessary and sufficient condition for the existence of an escaping radial profile $(\tf_\eps, g_\eps>0)$ in $(0,1)$ to the system \eqref{Eq:feegeeH1}--\eqref{Eq:20III21-feegeeBC}; we also prove uniqueness, minimality and monotonicity of the escaping radial profile. For that, in the context of $E_\eps$ defined over $\mcA$, we introduce the functional
\begin{align*}
I_{\eps}(f,g) 
	&= \frac{1}{|\Sphere^{N-1}|} E_{\eps}\bigg((f(r)\frac{x}{|x|}, 0_{\R^{M-N-1}}, g(r))\bigg)\\
	&= \frac{1}{2}\int_0^1 \Big[(f')^2 + (g')^2 + \frac{N-1}{r^2} f^2 + \frac{1}{\eps^2} W(1 - f^2 - g^2) \Big]\,r^{N-1}\,dr
\end{align*}
where $(f,g)$ belongs to
\be
\label{def_b}
\mcB = \Big\{(f,g): r^{\frac{N-1}{2}}f', r^{\frac{N-3}{2}} f, r^{\frac{N-1}{2}}g', r^{\frac{N-1}{2}}g \in L^2(0,1), f(1) = 1, g(1) = 0\Big\}.
\ee

The following result is reminiscent from Ignat-Nguyen \cite[Theorem 2.4]{IN} (for $\tilde W\equiv 0$). The proof of \cite[Theorem 2.4]{IN} is rather complicated (as it is proved for some general potentials $\tilde W$). We present here a simple proof that works in our context:

\begin{theorem}\label{Thm:ExtendedExist}
Let $2 \leq N \leq 6$, $M\geq N+1$, $W \in C^2((-\infty,1])$ satisfy \eqref{Eq:26VI18-E1} and be strictly convex.
Consider $\eps_N \in (0,\infty)$ in \eqref{epsN} such that $\ell(\eps_N)=0$.
Then the system \eqref{Eq:feegeeH1}--\eqref{Eq:20III21-feegeeBC} has an escaping radial profile $(\tf_{\eps},g_{\eps} )$ with $g_{\eps} > 0$ in $(0,1)$ if and only if $0 < \eps < \eps_N$. Moreover,  in the case $0 < \eps < \eps_N$, 

\begin{enumerate}

\item $(\tf_{\eps},g_{\eps}>0)$ is the unique escaping radial profile of \eqref{Eq:feegeeH1}--\eqref{Eq:20III21-feegeeBC} and $\frac{\tf_{\eps}}{r}, g_{\eps} \in C^2([0,1])$, $\tf_{\eps}^2 + g_{\eps}^2 < 1$, $\tf_{\eps} > 0$, $\tf_{\eps}' > 0$, $g_{\eps}' < 0$ in $(0,1)$; 

\item there are exactly two minimizers of $I_{\eps}$ in $\mcB$ given by $(\tf_{\eps},\pm g_{\eps} )$; 

\item the non-escaping radial profile $(f_\eps,0)$ is an unstable critical point of $I_{\eps}$ in $\mcB$ where $f_\eps$ is the unique radial profile in \eqref{1}. 

\end{enumerate}
\end{theorem}

Recall that for $\eps\geq \eps_N$, the non-escaping radial profile $(f_\eps,0)$ is the unique global minimizer  of $I_{\eps}$ in $\mcB$ (by Theorem \ref{eps_large} whose proof yields the minimality of $(u_\eps, 0_{\R^{M-N}})$ of $E_\eps$ in $\mcA$).

\begin{proof}[Proof of Theorem \ref{Thm:ExtendedExist}]
First, we focus on the existence of escaping radial profiles of \eqref{Eq:feegeeH1}--\eqref{Eq:20III21-feegeeBC}. Note that the direct method in calculus of variations implies that $I_{\eps}$ admits a minimizer $(\tf_{\eps},g_{\eps}) \in \mcB$. Since $(\tf_{\eps}, g_{\eps}) \in \mcB$, $(\tf_{\eps},g_{\eps}) \in C((0,1])$. It follows that $(\tf_{\eps},g_{\eps})$ satisfies \eqref{Eq:20III21-fee}--\eqref{Eq:20III21-feegeeBC} in the weak sense, and so $\tf_{\eps}, g_{\eps} \in C^2((0,1])$. Since $(|\tf_{\eps}|,|g_{\eps}|)$ is also a minimizer of $I_{\eps}$ in $\mcB$, the above argument also shows that $|\tf_{\eps}|, |g_{\eps}| \in C^2((0,1])$ satisfies \eqref{Eq:20III21-fee}--\eqref{Eq:20III21-feegeeBC}. Since $|\tf_{\eps}|,|g_{\eps}| \geq 0$ and $\tf_{\eps}(1) = 1$, the strong maximum principle yields $|\tf_{\eps}| > 0$ in $(0,1)$, and either $|g_{\eps}| > 0$ in $(0,1)$ or $g_{\eps} \equiv 0$ in $(0,1)$. It follows that $\tf_{\eps} > 0$ in $(0,1)$, and there are three alternatives: $g_{\eps} > 0$ in $(0,1)$,  $g_{\eps} < 0$ in $(0,1)$ or $g_{\eps} \equiv 0$ in $(0,1)$. Clearly, when $g_{\eps} \equiv 0$, $\tf_{\eps}$ is equal to the unique radial profile $f_\eps$  in \eqref{1}. By considering $(\tf_{\eps},-g_{\eps})$ instead of $(\tf_{\eps},g_{\eps})$ if necessary, we assume in the sequel that $g_{\eps} \geq 0$.

\medskip
\noindent {\it Claim:} if $0 < \eps < \eps_N$, then $g_{\eps} > 0$ in $(0,1)$ and $(f_\eps,0)$ is an unstable critical point of $I_{\eps}$ in $\mcB$.

\medskip
\noindent {\it Proof of Claim:} We define the second variation of $I_{\eps}$ at $(f_{\eps},0)$ as
\begin{align*}
Q_{\eps}(\alpha,\beta)&=\frac{d^2}{dt^2}\bigg|_{t=0} I_{\eps}\bigg((f_\eps,0)+t(\alpha, \beta)\bigg)\\
	&=  \int_{B^N} \Big[L_\eps \alpha \cdot \alpha + L_\eps \beta \cdot \beta
		+\frac{N-1}{r^2} \alpha^2  + \frac{2}{\eps^2} W''(1 - f_\eps^2) f_\eps^2 \alpha^2\Big]\,dx,
\end{align*}
for $\alpha, \beta\in C_c^\infty((0,1))$ which extends by density to the Hilbert space 
$$\mcH = \{(\alpha,\beta): (f_\eps + \alpha, \beta) \in \mcB\} \, \textrm{ with the norm} \quad \|(\alpha,\beta)\|_{\mcH} := \|(\alpha \frac{x}{|x|}, \beta)\|_{H^1(B^N,\RR^{N+1})}.$$ 
As $\eps\in (0,\eps_N)$, we have $\ell(\eps)<0$ by \eqref{epsN}. Taking $\beta \in H_0^1(B^N)$ to be any first eigenfunction of $L_{\eps}$ in $B^N$, which is radially symmetric, we have $r^{\frac{N-1}{2}}\beta', r^{\frac{N-1}{2}}\beta \in L^2(0,1)$, $\beta(1) = 0$ and 
$$Q_{\eps}(0,\beta)= \int_{B^N} L_\eps \beta \cdot \beta \, dx=\ell(\eps) \int_{B^N} \beta^2 \, dx< 0.$$ So, $(f_\eps,0)$ is an unstable critical point of $I_{\eps}$ in $\mcB$ if $\eps<\eps_N$. In particular, $(f_\eps,0)$ is not minimizing $I_{\eps}$ in $\mcB$ and therefore, by the above construction of the minimizer $(\tf_{\eps}, g_{\eps})$ of $I_{\eps}$ in $\mcB$, we deduce that $g_{\eps} > 0$. This proves the above Claim. 

Moreover, by 
\cite[Lemmas~2.7 and A.5, Proposition 2.9]{IN} (for $\tilde W\equiv 0$), we deduce that $\frac{\tf_{\eps}}{r}, g_{\eps} \in C^2([0,1])$, $\tf_{\eps}^2 + g_{\eps}^2 < 1$, $\tf_{\eps}' > 0$ and $g_{\eps}' < 0$ in $(0,1)$. 

\medskip

To conclude, we distinguish two cases:

\smallskip

\noindent {\it Case 1:} if $\eps\in (0, \eps_N)$, Claim yields the existence of an escaping radial profile $(\tf_{\eps},g_{\eps}>0)$. By \cite[Lemmas~2.7]{IN}, every escaping radial profile $(\tf_{\eps},g_{\eps}>0)$ is bounded (i.e., $\tf_{\eps}^2+g_{\eps}^2<1$ in $(0,1)$) and therefore, by Theorem \ref{thm:main}, the corresponding (bounded) escaping critical point $\tilde u_\eps$ in \eqref{Eq:feegeeH1}  is a global minimizer of $E_\eps$ over $\mcA$ and the set of minimizers of $E_\eps$ over $\mcA$ is then given by $\{R\tilde u_\eps \, :\, R\in O(M), \, Rp=p, \forall p\in \R^N\times \{0_{\R^{M-N}}\}\}$. Therefore, $(\tf_{\eps},\pm g_{\eps} )$ are the only two minimizers of $I_{\eps}$ in $\mcB$. In particular, this proves the uniqueness of the escaping radial profile $(\tf_{\eps},g_{\eps}>0)$. 

\smallskip

\noindent {\it Case 2:} if $\eps\geq \eps_N$, by the proof of Theorem \ref{eps_large}, the non-escaping vortex sheet solution $u_\eps(x)\equiv (f_\eps(|x|)\frac{x}{|x|}, 0_{\R^{M-N}})$ (by \eqref{identif}) is the unique minimizer of $E_\eps$ over $\mcA$. In particular, $(f_\eps, 0)$ is the unique minimizer of $I_{\eps}$ in $\mcB$, i.e., in the above construction of the minimizer $(\tf_{\eps}, g_{\eps})$ of $I_{\eps}$ in $\mcB$, we have $\tilde f_\eps=f_\eps$ and $g_\eps=0$ in $(0,1)$. We claim that no escaping radial profile $(\hat f_{\eps}, \hat g_{\eps}>0)$ exists if  $\eps\geq \eps_N$. Assume by contradiction that such an escaping radial profile $(\hat f_{\eps}, \hat g_{\eps}>0)$ exists. The same argument presented in Case 1 would imply that $(\hat f_{\eps}, \hat g_{\eps}>0)$ is a minimizer of $I_{\eps}$ in $\mcB$ which contradicts the uniqueness of the global minimizer $(f_\eps, 0)$.
\end{proof}

\subsection{Proof of Theorem \ref{thm:dicho}} 

We now prove the main result:

\begin{proof}[Proof of Theorem  \ref{thm:dicho}]
By Theorem \ref{Thm:ExtendedExist}, the existence of an escaping radially symmetric solution $\tilde u_\eps$ in \eqref{Eq:feegeeH1} is equivalent to $\eps\in (0, \eps_N)$. Moreover, in that case, the escaping radial profile $(\tilde f_\eps, g_\eps>0)$ is unique and bounded, i.e., $\tf_{\eps}^2+g_{\eps}^2<1$ in $(0,1)$.

\smallskip

\noindent {\it Case 1:} if $\eps\in (0, \eps_N)$, Theorem \ref{thm:main} implies that the (bounded) escaping radially symmetric critical point $\tilde u_\eps$ in \eqref{Eq:feegeeH1}  is a global minimizer of $E_\eps$ over $\mcA$ and every minimizer of $E_\eps$ over $\mcA$ has the form $R\tilde u_\eps$ for some orthogonal transformation $R\in O(M)$ keeping invariant the space 
$\R^N\times \{0_{\R^{M-N}}\}$. Moreover, by Theorem \ref{Thm:ExtendedExist}, the non-escaping radial profile $(f_\eps,0)$ is proved to be an unstable critical point of $I_{\eps}$ in $\mcB$, so the non-escaping vortex sheet solution $(u_\eps, 0_{\R^{M-N}})$ is an unstable critical point of $E_{\eps}$ in $\mcA$. 

\smallskip

\noindent {\it Case 2:} if $\eps\geq \eps_N$, the proof of Theorem \ref{eps_large} implies that the non-escaping radially symmetric vortex sheet solution $u_\eps(x)\equiv (f_\eps(|x|)\frac{x}{|x|}, 0_{\R^{M-N}})$ (by \eqref{identif}) is the unique minimizer of $E_\eps$ over $\mcA$. In this case, there is no bounded critical point $w_\eps$ of $E_\eps$ over $\mcA$ that escapes in some direction $e\in \mathbb{S}^{M-1}$; indeed, if such (bounded) escaping solution $w_\eps$ satisfying  \eqref{ass:Phi} exists, then Theorem \ref{thm:main} would imply that $w_\eps$ is a global minimizer of $E_\eps$ over $\mcA$ which contradicts that the non-escaping vortex sheet solution $u_\eps$ is the unique global minimizer of $E_\eps$ over $\mcA$.
\end{proof}

Theorem \ref{thm:dicho} holds also for the ``degenerate" dimension $n=0$. In this case, $\Omega=B^N$ and vortex sheets are vortex points, 
$$
E_\eps(u)=\int_{B^N} \Big[\frac{1}{2}|\nabla u|^2+\frac{1}{2\eps^2}W(1 - |u|^2)\Big]\,dx,
$$
$$
\mcA:=\{ u\in H^1(B^N; \R^M):\,  u(x)=(x, 0_{\R^{M-N}}) \textrm{ on } \partial B^N = \bS^{N-1}\}
$$
and radially symmetric vortex critical points of $E_\eps$ in $\mcA$ have the corresponding form in  \eqref{Eq:feegeeH1}:
\be
\label{u-inter}
\tilde u_\eps(x)=(\tf_{\eps}(r) \frac{x}{|x|}, 0_{\R^{M-N-1}}, g_{\eps}(r)) \in \mcA, \quad x\in B^N, r=|x|,
	\ee
where the radial profiles $(\tf_{\eps},g_{\eps})$ satisfy the system \eqref{Eq:20III21-fee}-\eqref{Eq:20III21-feegeeBC}
and are described in Theorem~\ref{Thm:ExtendedExist}; the non-escaping radially symmetric vortex solution is given here by
\be
\label{u11}
u_\eps(x)=(f_\eps(|x|)\frac{x}{|x|}, 0_{\R^{M-N}})  \quad \textrm{ for all } x\in B^N,
\ee
where the radial profile $f_\eps$ is the unique solution to \eqref{1}.
We obtain the following result
which generalizes  \cite[Theorem 1.1]{INSZ_ENS} that was proved in the case $N=2$ and $M=3$ (without identifying the meaning of the dichotomy parameter $\eps_N$ in \eqref{epsN}). 

\begin{theorem}\label{thm:ball}
Let $2 \leq N \leq 6$, $M\geq N+1$, $\Omega=B^N$,  $W \in C^2((-\infty,1])$ satisfy \eqref{Eq:26VI18-E1} and be strictly convex.
Consider $\eps_N \in (0,\infty)$ such that $\ell(\eps_N)=0$ in \eqref{epsN}.
Then there exists an escaping radially symmetric vortex solution $\tilde u_\eps$ in \eqref{u-inter} with the radial profile $(\tf_{\eps},g_{\eps}>0)$ given in Theorem \ref{Thm:ExtendedExist} if and only if $0 < \eps < \eps_N$. Moreover,   

\begin{enumerate}

\item if $0 < \eps < \eps_N$, $\tilde u_\eps$ is a global minimizer of $E_\eps$ in $\mcA$ and all global minimizers of  $E_\eps$ in $\mcA$ are radially symmetric given by $R\tilde u_\eps$ where $R\in O(M)$ is an orthogonal transformation of $\R^M$ satisfying $Rp=p$ for all $p\in \R^N\times \{0_{\R^{M-N}}\}$. In this case, the non-escaping vortex solution $u_\eps$ in \eqref{u11} is an unstable critical point of $E_{\eps}$ in $\mcA$. 

\item if $\eps\geq \eps_N$, the non-escaping vortex solution $u_\eps$ in \eqref{u11} is the unique global minimizer of $E_\eps$ in $\mcA$. Furthermore, there are no bounded critical points $w_\eps$ of $E_{\eps}$ in $\mcA$ that escape in a direction $e\in \bS^{M-1}$, i.e., $w_\eps\cdot e>0$ a.e. in $\Omega$.

\end{enumerate}
\end{theorem}

The proof follows by the same argument used for Theorem  \ref{thm:dicho}, the main difference is that in the ball $\Omega=B^N$, a critical point $w_\eps$ of $E_{\eps}$ in $\mcA$ satisfies the PDE system with Dirichlet boundary condition (instead of the mixed Dirichlet-Neumann condition in \eqref{E-L-M}):
\begin{align*}
-\Delta w_\eps&=\frac1{\eps^2} w_\eps \, W'(1-|w_\eps|^2)\quad \textrm{ in } \, B^N,\\
w_\eps(x)&=(x, 0_{\R^{M-N}}) \quad \textrm{ on } \, \partial B^N.
\end{align*}

\appendix
\section{Appendix. Vortex sheet $\bS^{M-1}$-valued harmonic maps in cylinders}

In dimensions $M> N\geq 2$ and $n\geq 1$, for the cylinder shape domain $\Omega=B^N\times (0,1)^n$, we consider the harmonic map problem for $\bS^{M-1}$-valued maps $u\in H^1(\Omega; \bS^{M-1})\cap \mcA$ associated to the Dirichlet energy
$$E(u)=\frac12\int_{\Omega} |\nabla u|^2\, dx dz.$$ Any critical point $u:\Omega \to \bS^{M-1}$ of this problem satisfies
\be
\label{HMP}
\left\{\begin{array}{l}
-\Delta u=u \, |\nabla u|^2\quad \textrm{ in } \, \Omega,\\
\frac{\partial u}{\partial z}=0 \quad \textrm{ on } \, B^N\times \partial (0,1)^n,\\
u(x,z)=(x, 0_{\R^{M-N}}) \quad \textrm{ on } \, \partial B^N\times (0,1)^n.
\end{array}
\right.
\ee
We will focus on radially symmetric vortex sheet $\bS^{M-1}$-valued harmonic maps having the following form (invariant in $z$-direction): 
\begin{equation}
u(x,z)=(f(r) \frac{x}{|x|}, 0_{\R^{M-N-1}}, g(r)) \in \mcA, \quad  x\in B^N, z\in (0,1)^n, r=|x|,
	\label{Eq:hmp}
\end{equation}
where the radial profile $(f,g)$ satisfies \be
\label{f2g2}
f^2+g^2=1 \quad \textrm{in} \quad (0,1),
\ee
and the system of ODEs:
\begin{align}
-f'' - \frac{N-1}{r}  f' + \frac{N-1}{r^2} f
	&=\Gamma(r) f  \quad \textrm{in} \quad (0,1),
	\label{Eq:MM-fee}\\
-g'' - \frac{N-1}{r} g' 
	&= \Gamma(r) g  \quad \textrm{in} \quad (0,1),
	\label{Eq:MM-gee}\\
f(1) &= 1 \text{ and } g(1) = 0,
	\label{Eq:MM-feegeeBC}
\end{align}
where 
$$
\Gamma(r)=(f')^2+\frac{N-1}{r^2}f^2+(g')^2$$ is the Lagrange multiplier due to the unit length constraint in \eqref{f2g2}. As for the Ginzburg-Landau system, we distinguish two type of radial profiles:

\smallskip

$\bullet$ the {\it non-escaping} radial profile $(\bar f\equiv 1, \bar g\equiv 0)$ yielding the {\it non-escaping} (radially symmetric) vortex sheet $\bS^{M-1}$-valued harmonic map (also called ``equator" map):
\be
\label{ecuato}
\bar u(x,z)=(\frac{x}{|x|}, 0_{\R^{M-N}})   \quad  x\in B^N, z\in (0,1)^n.
\ee
Note that $\bar u$ is singular and the singular set of this map is the vortex sheet $\{0_{\R^{M-N}}\}\times (0,1)^n$ of dimension $n$ in $\Omega$. Also, observe that $\bar u\in H^1(\Omega, \bS^{M-1})$ if and only if $N\geq 3$. 
\smallskip

$\bullet$ the {\it escaping} radial profile $(f, g)$ with $g>0$ in $(0,1)$; in this case, it holds $f(0)=0$, $g(0)=1$ and we say that  $u$ in \eqref{Eq:hmp} is an {\it escaping} (radially symmetric) vortex sheet $\bS^{M-1}$-valued harmonic map. Note that $u$ is smooth for every dimension $M> N\geq 2$ and $n\geq 1$ and the zero set of $(u_1, \dots, u_N)$ is the vortex sheet $\{0_{\R^{M-N}}\}\times (0,1)^n$ of dimension $n$ in $\Omega$. Obviously, $(f, -g<0)$ is another radial profile satisfying \eqref{f2g2}-\eqref{Eq:MM-feegeeBC}.

\smallskip

The properties of such radial profiles are proved in \cite{JagKaul} (see also \cite[Theorem 2.6]{IN} for $\tilde W\equiv 0$ in those notations). More precisely,  
\begin{enumerate}[(a)]
\item If $N \geq 7$, the non-escaping radial profile $(\bar f\equiv 1, \bar g\equiv 0)$ is the unique minimizer of 
\begin{align*}
I(f,g)
	&= \frac{1}{|\Sphere^{N-1}|} E\bigg((f(r) \frac{x}{|x|}, 0_{\R^{M-N-1}}, g(r))\bigg)=
	\frac12\int_0^1 \Big[(f')^2 + (g')^2 + \frac{N-1}{r^2} f^2 \Big]\,r^{N-1}\,dr,
\end{align*}
where $(f,g)$ belongs to
$\mcB\cap \big\{(f,g)\, :\,  f^2 + g^2 = 1\big\}$ with $\mcB$ defined in \eqref{def_b}. Moreover, the system \eqref{f2g2}--\eqref{Eq:MM-feegeeBC} has no escaping radial profile $(f, g)$ with $g > 0$ in $(0,1)$.

\item If $2 \leq N \leq 6$, then there exists a unique escaping radial profile $(f, g)$ with $g > 0$ satisfying \eqref{f2g2}--\eqref{Eq:MM-feegeeBC}.  Moreover, $(f, \pm g)$ are the only two global minimizers of  $I$ in $\mcB\cap \big\{(f,g)\, :\,  f^2 + g^2 = 1\big\}$, $\frac{f}{r}, g \in C^\infty([0,1])$, $f(0)=0$, $g(0)=1$, $f > 0$, $f' > 0$ and $g' < 0$ in $(0,1)$. In addition, for $3 \leq N \leq 6$, the non-escaping solution $(\bar f\equiv1, \bar g\equiv 0)$ is an unstable critical point of $I$ in $\mcB\cap \big\{(f,g)\, :\,  f^2 + g^2 = 1\big\}$.\footnote{For $N = 2$, $(1,0)\notin \mcB$; however, we can define the second variation of $I$ at $(1,0)$ along directions $(0,q)$ compactly supported in $(0,1)$:
\[
Q(0,q) = \int_0^1 \Big[(q')^2 - \frac{N-1}{r^2}q^2 \Big]\,r^{N-1}\,dr,
\]
and one can prove the existence of $q \in Lip_c(0,1)$ such that
$
Q(0,q) < 0$ (see e.g. \cite[Remark 2.16]{IN}).} 

\end{enumerate} 

There is a large number of articles studying existence, uniqueness, regularity and stability of radially symmetric  $\mathbb{S}^{M-1}$-valued harmonic maps (e.g., \cite{JagKaul1, JagKaul, SU, SU1, SS1, Lin-Wang, INSZ_ENS}). We summarize here the main result for our problem in the cylinder shape domain $\Omega=B^N\times (0,1)^n$: if $N\leq 6$, then minimizing $\mathbb{S}^{M-1}$-valued harmonic maps in $\mcA$ are smooth, radially symmetric and escaping in one-direction; if $N\geq 7$, then there is a unique minimizing $\mathbb{S}^{M-1}$-valued harmonic map in $\mcA$ which is singular and given by the equator map $\bar u$ in \eqref{ecuato}. \footnote{We mention the paper of Bethuel-Brezis-Coleman-H\'elein \cite{BBCH} about a similar phenomenology in a domain $\Omega=(B^2\setminus B_\rho)\times (0,1)\subset \RR^3$ where $B_\rho\subset \RR^2$ is the disk centered at $0$ of radius $\rho$. }

\begin{theorem}\label{thm:dico-HMP}
Let $n\geq 1$, $N\geq 2$, $M\geq N+1$ and $\Omega=B^N\times (0,1)^n$. Then

\begin{enumerate}

\item if $2\leq N\leq 6$, then the {\it escaping} radially symmetric vortex sheet solution $u$ in \eqref{Eq:hmp} with $g>0$ is a minimizing $\bS^{M-1}$-valued harmonic map in $\mcA$ and all minimizing $\bS^{M-1}$-valued harmonic maps in $\mcA$  are smooth radially symmetric given by $R u$ where $R\in O(M)$ satisfies $Rp=p$ for all $p\in \R^N\times \{0_{\R^{M-N}}\}$. In this case, the equator map $\bar u$ in \eqref{ecuato} is an unstable $\bS^{M-1}$-valued harmonic map in $\mcA$. 

\item if $N\geq 7$, the non-escaping vortex sheet solution $\bar u$ in \eqref{ecuato} is the unique minimizing $\bS^{M-1}$-valued harmonic map in $\mcA$. Moreover, there is no  $\bS^{M-1}$-valued harmonic map $w$ in $\mcA$ escaping in a direction $e\in \bS^{M-1}$, i.e., $w\cdot e>0$ a.e. in $\Omega$.

\end{enumerate}

\end{theorem}

The main ingredient is the following result yielding minimality of escaping ${\mathbb S}^{M-1}$-valued harmonic maps. This  is reminiscent from Sandier-Shafrir \cite{SS1} (see also \cite[Theorem 1.5]{INSZ_ENS}). 

\begin{theorem}\label{prop:HMP}
Let $n\geq 1$, $M> N \geq 2$ and $\Omega=B^N\times (0,1)^n$.
Assume that 
$w \in \mcA \cap  H^1(\Omega, {\mathbb S}^{M-1})$ is a ${\mathbb S}^{M-1}$-valued harmonic map satisfying \eqref{HMP} and
\be\label{ass:Phi+}
w \cdot e>0\textrm{ a.e. in }\Omega
\ee 
in an escaping direction $e\in \bS^{M-1}$.
Then $w$ is a minimizing ${\mathbb S}^{M-1}$-valued harmonic map in $\mcA$ and all minimizing ${\mathbb S}^{M-1}$-valued harmonic maps in $\mcA$ are of the form $Rw$ where $R\in O(M)$ is an orthogonal transformation of $\R^M$ satisfying $Rp=p$ for all $p\in \R^N\times \{0_{\R^{M-N}}\}$.
\end{theorem}

\begin{proof}[Proof of Theorem \ref{prop:HMP}] We give here a simple proof based on the argument in \cite{INSZ_ENS} that avoids the regularity results used in \cite{SS1}. By the $H^{1/2}$-trace theorem applied for $w\in H^1(\Omega, \bS^{M-1})$, \eqref{ass:Phi+} implies that $w \cdot e\geq 0$ on $\partial B^N\times (0,1)^n$. Combined with the vortex boundary condition in \eqref{HMP}, we deduce that the escaping direction $e$ has to be orthogonal to $\R^N\times \{0_{\R^{M-N}}\}$ and up to a rotation, we can assume that $e=e_M$ (as in \eqref{def-enu}). Then $\phi=w \cdot e_M>0$ a.e. in $\Omega$ satisfies 
\be
\label{har}
-\Delta \phi=|\nabla w|^2\phi \ \hbox{  in } \ \Omega, \, \, \frac{\partial \phi}{\partial z}=0 \, \textrm{ on } \, B^N\times \partial (0,1)^n,
 \, \, \phi=0 \, \textrm{ on } \, \partial B^N\times (0,1)^n.
\ee
We consider configurations\footnote{Note that for any $\tilde w \in \mcA \cap  H^1(\Omega, {\mathbb S}^{M-1})$, the map $\tilde w-w$ has an extension in $H^1_0(B^N\times \R^n,\R^M)$.} $\tilde w=w+v:\Omega\to \bS^{M-1}$ with $v\in H^1_0(B^N\times \R^n,\R^M)$ (in particular, $|v|\leq 2$ in $\Omega$). Then 
\be\label{const:uv}
 2w\cdot v+|v|^2=0 \quad \textrm{ a.e. in $\Omega$. }
\ee  Using \eqref{HMP} and \eqref{const:uv}, we obtain
$$
2\int_\Omega \nabla w\cdot \nabla v=2\int_\Omega |\nabla w|^2 w\cdot v\,dx=-\int_\Omega |\nabla w|^2 |v|^2\,dx,
$$
yielding\footnote{Note that the functional $Q$ represents the second variation of $E$ at $w$, but here the map $v$ is not necessarily orthogonal to $w$.}
\be
\label{expd}
\int_\Omega |\nabla (w+v)|^2\,dx-\int_\Omega |\nabla w|^2\,dx =\int_\Omega|\nabla v|^2- |\nabla w|^2|v|^2\,dx=:Q(v).
\ee
To show that $w$ is minimizing, we prove that $Q(v)\ge 0$ for all $v\in H^1_0(B^N\times \R^n,\R^M)\cap L^\infty(\Omega;\RR^M)$ (note that this is a class larger than what we need, as we do not require that $v$ satisfy the pointwise constraint \eqref{const:uv}).
For that, we take an arbitrary map $\tilde v\in C_c^\infty(B^N\times \R^n,\R^M)$ of support $\omega$ and decompose it as  $\tilde v = \phi \Psi$ in $\Omega$. This decomposition makes sense as $\phi\geq \delta>0$ in $\omega\cap \Omega$ for some $\delta>0$ (which may depend on $\omega$). Indeed, 
by \eqref{ass:Phi+} and \eqref{har}, $\phi$ is a superharmonic function (i.e., $-\Delta \phi\geq 0$ in $\Omega$) that belongs to $H^1(\Omega)$. As $\frac{\partial \phi}{\partial z}=0$ on $B^N\times \partial (0,1)^n$, $\phi$ can be extended by even mirror symmetry to the domain $\tilde \Omega=B^N\times (-1, 2)^n$ so that $\phi$ is superharmonic in $\tilde \Omega$. Thus, the weak Harnack inequality (see e.g. \cite[Theorem 8.18]{GT}) implies that on the compact set $\omega\cap \Omega$ in $\tilde \Omega$, we have $\phi\geq \delta>0$ for some $\delta$. So,  $\tilde v = \phi \Psi$ in $\Omega$ with
$\Psi=(\Psi_1, \dots, \Psi_M)\in H^1\cap L^\infty(\Omega;\RR^M)$ vanishing in a neighborhood of $\partial B^N\times (0,1)^n$. Then  integration by parts yields for $1 \leq j \leq M$:
\begin{align*}
Q(\tilde v_j)&=\int_\Omega |\nabla \tilde v_j|^2- |\nabla w|^2  \phi \cdot \phi \Psi_j^2\,dx\\
	&\stackrel{\eqref{har}}{=} \int_\Omega  |\nabla (\phi \Psi_j)|^2-\nabla \phi \cdot \nabla (\phi\,\Psi_j^2) \,dx
	= \int_\Omega \phi^2 |\nabla \Psi_j|^2\,dx \ge 0 
	\end{align*}
for all $\tilde v\in C_c^\infty(B^N\times \R^n,\R^M)$. 
Then for every $v\in H^1_0(B^N\times \R^n,\R^M)\cap L^\infty(\Omega;\RR^M)$, there exists a sequence $\tilde v^k\in C_c^\infty(B^N\times \R^n,\R^M)$  such that $\tilde v^k\to v$ and $\nabla \tilde v^k\to \nabla v$ in $L^2$ and a.e. in $B^N\times \R^n$ and $|\tilde v^k|\leq \|v\|_{L^\infty(\Omega)}+1$ in $\Omega$ for every $k$. In particular, by dominated convergence theorem, we have 
$Q(\tilde v^k)\to Q(v)$ 
thanks to \eqref{expd}. Thus, we deduce that for every compact $\omega\subset \tilde \Omega=B^N\times (-1, 2)^n$,
$$Q(v)=\lim_{k\to \infty} Q(\tilde v^k)\geq \liminf_{k\to \infty} \int_{\omega \cap \Omega} \phi^2 |\nabla \big(\frac{\tilde v^k}\phi\big)|^2\,dx\geq \int_{\omega \cap \Omega} \phi^2 |\nabla \big(\frac{v}\phi\big)|^2\,dx\geq 0,$$ where we used Fatou's lemma.
In particular, $w$ is a minimizing ${\mathbb S}^{M-1}$-valued harmonic map by \eqref{expd} and $Q(v)=0$ yields the existence of a vector $\lambda\in \R^M$ such that $v=\lambda \phi$ a.e. in $\Omega$. Then the classification of the  minimizing ${\mathbb S}^{M-1}$-valued harmonic maps follows by \eqref{const:uv} as in the Step 3 of the proof of Theorem~\ref{thm:main}. \end{proof}

\begin{proof}[Proof of Theorem \ref{thm:dico-HMP}]
{\it 1.} This part concerning the dimension $2\leq N\leq 6$ follows from Theorem \ref{prop:HMP} and the instability of the radial profile $(1,0)$ for $I$ in $\mcB\cap \big\{(f,g)\, :\,  f^2 + g^2 = 1\big\}$ as explained above.

{\it 2.} This part for dimension $N\geq 7$ follows the ideas in \cite{JagKaul}. More precisely, calling $X=(x,z)$ the variable in $\Omega$, we have 
as in the proof of Theorem \ref{prop:HMP} for every 
$v\in H^1_0(B^N\times \R^n,\R^M)$ with $|v+\bar u|=1$ in $\Omega$:
\begin{align*}
\int_\Omega |\nabla (\bar u+v)|^2\,dX-&\int_\Omega |\nabla \bar u|^2\,dX =\int_\Omega \big(|\nabla v|^2- |\nabla \bar u|^2|v|^2\big)\,dX\\
&=\int_{\Omega} |\nabla_z v|^2\, dX+\int_{(0,1)^n} \, dz \int_{B^N} \big(|\nabla_x v|^2-\frac{N-1}{|x|^2} |v|^2\big)\, dx\\
&\geq \int_{\Omega} |\nabla_z v|^2\, dX+\bigg(\frac{(N-2)^2}{4}-(N-1)\bigg)\int_{\Omega} \frac{|v|^2}{|x|^2}\, dX\geq 0
\end{align*}
where we used the Hardy inequality for $v(\cdot, z)\in H^1_0(B^N, \R^M)$ for a.e. $z\in (0,1)^n$. This proves that $\bar u$ is the unique minimizing $\bS^{M-1}$-valued harmonic map in $\mcA$. Combined with Theorem \ref{prop:HMP}, we conclude that there is no escaping $\bS^{M-1}$-valued harmonic map $w$ in $\mcA$.
\end{proof}

\end{document}